\documentclass[12pt,reqno]{amsart}
\usepackage{hyperref}
\usepackage[latin1]{inputenc}
\usepackage[all]{xy}

\newtheorem{theorem}{Theorem}
\newtheorem{proposition}[theorem]{Proposition}
\newtheorem{lemma}[theorem]{Lemma}

\newtheorem{remark}[theorem]{Remark}

\newcommand{\R}{{\mathbb R} }

\newcommand{\wt}{\widetilde}

\def\pt{\partial}
\def\k{\mathfrak{k}}
\def\g{\mathfrak{g}}

\begin{document}

\title{Symplectic reduction of Sasakian manifolds}

\author[I.\ Biswas]{Indranil Biswas}

\address{School of Mathematics, Tata Institute of Fundamental
Research, Homi Bhabha Road, Mumbai 400005, India}

\email{indranil@math.tifr.res.in}

\author[G.\ Schumacher]{Georg Schumacher}

\address{Fachbereich Mathematik und Informatik,
Philipps-Universität Marburg, Lahnberge, Hans-Meerwein-Strasse, D-35032
Marburg, Germany}

\email{schumac@mathematik.uni-marburg.de}



\date{}

\begin{abstract}
When a complex semisimple group $G$ acts holomorphically on a K\"ahler manifold $(X,\, \omega)$
such that a maximal compact subgroup $K\, \subset\, G$ preserves the
symplectic form $\omega$, a basic result of symplectic geometry says that the corresponding categorical
quotient $X/G$ can be identified with the quotient of the zero-set of the moment map by the action of $K$.
We extend this to the context of a semisimple group acting on a Sasakian manifold.
\end{abstract}

\maketitle

\tableofcontents

\section{Introduction}

Contact manifolds can be thought of as odd dimensional analogs of symplectic manifolds.
In the same way, Sasakian manifolds can be regarded as odd dimensional analog of
K\"ahler manifolds. These manifolds were introduced by Sasaki \cite{Sa1}, \cite{Sa2}, \cite{SH}.
The topic remained dormant for more than thirty years until the following things happened:
\begin{enumerate}
\item In the AdS/CFT correspondence discovered by J. Maldacena \cite{Ma} it was realized
that Sasakian manifolds play a key role in string theory. Over time, many works in this direction
emerged (see \cite{Ma1}, \cite{Ma2}, \cite{Ma3} and references therein).

\item C.P. Boyer and K. Galicki worked systematically and published a series of papers investigating
various differential geometric aspects of Sasakian manifolds (see \cite{BGsusy} and references
therein).
\end{enumerate}

Let $(X,\, \omega)$ be a compact K\"ahler manifold, and let $G$ be a complex semisimple affine algebraic group
acting holomorphically on $X$ such that the action of a maximal compact
subgroup $K\, \subset\, G$ preserves the K\"ahler form $\omega$. Let $\mu\, :\, X\, \longrightarrow\, \text{Lie}(K)^*$
be the moment map for this action.
It is known that the categorical quotient $X/G$ is identified with the quotient $\mu^{-1}(0)/K$; the reader
is referred to \cite{kirwan} (see also \cite{he-lo}, \cite{snow}).

Here we take a Sasakian manifold
$(X ,\, g ,\, \xi)$; let $(M,\, \omega_M)$ be the associated K\"ahler manifold whose underlying manifold
is $X\times {\mathbb R}_+$. Let $r$ denote the standard coordinate on $ {\mathbb R}_+$.
Let $G$ be a complex semisimple affine algebraic group acting holomorphically on $M$
such that the action of a maximal compact subgroup $K\, \subset\, G$ preserves $X$. The action of $K$
preserves the contact one-form. We also assume the following:
\begin{itemize}
\item $\xi$ is orthogonal to the orbits of the action of $K$,

\item ${\pt}/{\pt r}$ is orthogonal, with respect to $\omega_M$,
to the distribution on $M$ given by ${\rm Lie}(K)$, and

\item $[\k,\,{\pt}/{\pt r}]\,=\,0$ (we denote by $\k$ the distribution on $M$ given by ${\rm Lie}(K)$.
\end{itemize}

Using the relationship between the Sasakian manifolds and K\"ahler manifolds, we prove that
$\mu^{-1}(0)/K$ is a Sasakian manifold, where $\mu$ as before is the moment map, such that the
categorical quotient $M/G$ is the K\"ahler manifold associated to it
(Theorem \ref{thm-m}).

As explained in the Acknowledgements, a similar result was proved earlier in \cite{GO}. Methods employed here
differ from that of \cite{GO}.

\section{Sasakian manifolds}

We denote by $(X,\,g)$ a connected, oriented Riemannian manifold equipped with the corresponding Levi-Civita connection
$\nabla$. It is called {\em Sasakian} if the metric cone
$$
(X \times \R_+, \, dr^2 \oplus r^2 g)
$$
is {\em K\"ahler}. We denote by $J$ the complex structure, and
\begin{equation}\label{eq:xi}
\xi = \left.J\left(\frac{\pt}{\pt r}\right)\right|_{X\times\{1\}}
\end{equation}
is called the {\em Reeb} vector field, where $X\times\{1\}$ is identified with $X$.

The computation of the Nijenhuis torsion tensor in terms of the Reeb vector field leads to the following well-known 
characterization of a Sasakian manifold.

\begin{theorem}[{\cite[Definition-Theorem 10]{BGsusy}}]\label{de:sasaki}
The following conditions for a Riemannian manifold $(X,g)$ are equivalent:
\begin{enumerate}
\item[(i)] There is a Killing vector field $\xi$ on $X$ of unit length such that the section
\begin{equation}\label{Phi}
\Phi \,\in \, C^\infty(X,\, TX\otimes (TX)^*)
\end{equation}
defined by $v\, \longmapsto\, -\nabla_v\xi$, $v\, \in\, TX$, satisfies the following identity
for the Lie derivative of $\Phi$, which
is defined by $(\nabla_v\Phi)(w) \,= \,\nabla_v(\Phi(w))- \Phi(\nabla_v(w))$:
\begin{equation}\label{id.}
(\nabla_v \Phi) (w) \,=\, g(v ,\,w)\xi- g(\xi ,\,w)v
\end{equation}
for all $v,\,w \,\in\, T_xX$ and all $x \,\in \, X$.

\item[(ii)] There is a Killing vector field $\xi$ on $X$ of unit length
such that the Riemann curvature tensor $R$ of $(X,\,g)$ satisfies the identity
$$
R(v,\xi)w\,=\, g(\xi ,\,w)v- g(v ,\,w)\xi
$$
for all $v$ and $w$ as above.

\item[(iii)] The metric cone $(X\times{\mathbb R}_+,\, dr^2 \oplus r^2g)$ is K\"ahler.
\end{enumerate}
\end{theorem}
We will point out some facts regarding the above equivalent conditions.

Given a Killing vector field $\xi$ of unit length satisfying condition (i), the K\"ahler structure on
${\mathbb R}_+\times X$ asserted in statement (iii) is constructed as follows. Let $F$ be the distribution of $X$
of rank $2n$ given by the orthogonal complement of $\xi$. The homomorphism $\Phi$ (defined in \eqref{Phi}) preserves
the above defined distribution $F$ on $X$, since $g(\Phi(v),\xi)\,=\,\frac{1}{2}v(g(\xi,\xi))\,=\,0$, and furthermore,
\begin{equation}\label{e2}
(\Phi\vert_F)^2 \, =\, -\text{Id}_F \, .
\end{equation}
Then an almost complex structure $J$ on ${\mathbb R}_+ \times X$ defined by the following conditions:
\begin{equation}
{J}\vert_F \,=\, \Phi\vert_F
\end{equation}
satisfying \eqref{eq:xi} and the corresponding equation
\begin{equation}
{J}(\xi) \,=\, -\frac{d}{dr}\, .
\end{equation}
The almost complex structure $J$ is in fact the complex structure stated in (iii). Condition \eqref{id.} is equivalent to the vanishing of the Nijenhuis tensor, and the Riemannian metric $dr^2 \oplus r^2g$ on ${\mathbb R}_+\times X$
is K\"ahler with respect to $J$.

Conversely, if the metric cone $(X \times{\mathbb R}_+ ,\, dr^2 \oplus r^2g)$ is K\"ahler, then consider the vector field $\xi$ 
given by \eqref{eq:xi}, where $J$ is the almost complex structure on $X\times \R_+$. The vector field $\xi$ defined this 
way satisfies condition $(i)$, which is known to be equivalent to $(ii)$ in Definition~\ref{de:sasaki}.

In this sense, the vector field $\xi$ (or equivalently, the K\"ahler structure on $X\times\R_+$) can and will be considered as part of the definition of a Sasakian manifold to be denoted by $(X,g,\xi)$.

Let $X$ be a smooth oriented Riemannian manifold of dimension $2n+1$ and $F\,\subset\,TX$
an oriented smooth distribution of rank $2n$. The quotient map
$$
TX \,\longrightarrow\, TX/F\,=:\, N
$$
defines a smooth one-form on $X$
\begin{equation}\label{e1}
\omega\, \in \, C^\infty(X,\, T^*X\otimes N)
\end{equation}
with values in the line bundle $N$. Since $X$ is oriented, the orientation of
$F$ induces an orientation of the normal bundle $N$. Therefore, $N$ has a
canonical smooth section given by the positively oriented vectors of unit
length in the fibers of $N$. Consequently, the form $\omega$ in \eqref{e1}
gives a nowhere vanishing smooth one-form on $X$. This one-form will also be
denoted by $\omega$. The distribution $F$ is said to be a \textit{contact
structure} on $X$ if the $(2n+1)$-form $(d\omega)^n\wedge \omega$ is nowhere
vanishing. (See \cite{BGsusy} and references therein.)

\begin{remark}\label{rem1}
{\rm The distribution $F$ is integrable, if it satisfies the Frobenius
condition which says that the one-forms $\omega$ satisfy the condition
$(d\omega)\wedge \omega \,=\, 0$. Therefore, a contact structure $F$ is not
integrable.}
\end{remark}

Now let $(X ,\, g ,\, \xi)$ be a Sasakian manifold. The distribution $F$ on $X$
of rank $2n$ that is given by the orthogonal complement of the Killing vector field
$\xi$ defines a contact structure on $X$. We note that the corresponding
one-form $\omega$ is the dual of $\xi$ with respect to the metric $g$, i.e.\ $\omega(u)
\,=\,g(\xi,u)$. From the condition that $(d\omega)^n\wedge \omega$ is nowhere vanishing it follows
that the restriction of $d\omega$ to $F$ is fiberwise nondegenerate.

\begin{lemma}\label{lem0}
For all $x \,\in \,X$ and all $v,\,w\, \in\, F_x$,
\begin{equation}\label{eq:deomega}
d\omega (v ,\,w) \,=\, - g(\Phi (v) ,\,w)\, ,
\end{equation}
where $\Phi$ is defined in \eqref{Phi}.
\end{lemma}

\begin{proof}
From the definition of $\Phi$,
$$
- g(\Phi (v) ,\,w)\,= \, g(\nabla_v\xi ,\,w)\, .
$$
Since $\xi$ is a Killing vector field,
\begin{equation}\label{KVF}
g(\nabla_v\xi ,\,w) + g(\nabla_w\xi ,\,v)\,=\, 0\, .
\end{equation}

Extend $v$ and $w$ to smooth sections $\widetilde{v}$ and
$\widetilde{w}$ of $F$. Since $\widetilde{w}$ is orthogonal to $\xi$,
$$
g(\nabla_v\xi,\,w)\,=\, - g(\xi ,\,\nabla_v \widetilde{w})\, .
$$
Using \eqref{KVF},
$$
g(\nabla_v\xi,\,w) \,=\, - g(\nabla_w\xi,\,v) \,=\,
g(\xi,\,\nabla_w \widetilde{v})
$$
because $\widetilde{v}$ is also orthogonal to $\xi$. Therefore,
$$
- g(\Phi (v) ,\,w) \,=\, \frac{1}{2}(- g(\xi ,\,\nabla_v
\widetilde{w})+
g(\xi ,\,\nabla_w \widetilde{v}))\,=\, -\frac{1}{2}g(\xi ,\,
[\widetilde{v} ,\,\widetilde{w}])\, .
$$
But $-\frac{1}{2}g(\xi ,\,
[\widetilde{v} ,\,\widetilde{w}]) \,= \,d\omega(v,\,w)$
because both $\widetilde{v}$ and $\widetilde{w}$ are orthogonal
to $\xi$.
\end{proof}

We now consider $N$ as the subbundle of $TX$ generated by $\xi$.

\section{Symplectic quotients of K\"ahler manifolds}\label{se:symq}

We will summarize some basic facts, which can be found in Kirwan's work \cite{kirwan}.

Let $(M,\,\omega_M)$ be a compact K\"ahler manifold acted on holomorphically by a complex Lie group $G$, which is the 
complexification of a maximal compact subgroup $K$. Assume that the K\"ahler form is preserved by the action of $K$, meaning 
$k^*\omega_M\,=\,\omega_M$ for all $k\,\in\, K$.

Furthermore we assume the existence of a {\em moment map}
$$
\mu\,:\,M \,\longrightarrow\, \k^*
$$
for the underlying symplectic manifold, where $\k\,=\,{\rm Lie}(K)$.

By definition, a moment map for the action of $K$ on $(M,\omega_M)$ is $K$-equivariant with respect to the action of $K$ on $M$ and the co-adjoint action $Ad^*$ of $K$ on $\k^*$ satisfying the following \def\k{\mathfrak{k}} condition.
Note first that for any $a\in \k$ the composition $d\mu\,:\,TM \,\longrightarrow\, \k^*$ with the evaluation at $a$ defines a $1$-form on $M$. This form is required to correspond under the duality defined by $\omega_M$ to the vector field on $M$ that is induced by $a$: For all $x\in M$ and for all $\xi \in T_xM$
\begin{equation}\label{eq:mm}
 d\mu(x)(\xi)\cdot a \,=\, \omega_M(\xi,a)\, ,
\end{equation}
where $\cdot$ denotes the natural pairing of $\k$ and $\k^*$.

The Marsden-Weinstein theorem states that a moment map exists, and is uniquely determined, if the group $K$ is semisimple \cite{MaWe}. Furthermore it is known to exist, if $H^1(\g)=0$, and $H^2(\g)=0$ (cf.\ \cite[Section 26]{Gu-St}). A moment map is explicitly given for the action of $U(N+1)$ on the complex projective space $\mathbb P_N$ equipped with the Fubini-Study form. In particular, if a compact group with a $U(N+1)$-representation acts on a projective variety $M\subset \mathbb P_N$, a moment map exists.

If the action of a connected reductive group $G$ on a projective manifold $M$ lifts to a (very) ample line bundle 
$L$, like in the above explicit case, the notion of stable and semistable points from geometric invariant theory 
for the group action hold. The loci $M^s \subset M^{ss}\subset M$ of stable and semistable points are known to be 
Zariski open in $M$, and the geometric and categorical quotients $M^s/G \subset M^{ss}/\!\!/ G$ exist as projective 
varieties.

Kirwan showed in \cite{kirwan} that Mumford's geometric invariant
theoretic quotient $M^{ss}/\!\!/G$ coincides with the Marsden-Weinstein quotient 
$\mu^{-1}(0)/K$ in the projective case under the assumption that the group $G$ is semisimple with finite stabilizers, and 
that $K$ acting on $\mu^{-1}(0)$ so that all points are stable. As a result, the quotient has the structure of a complex 
orbifold. The set $\mu^{-1}(0)$ is a differentiable manifold under this assumption because of \eqref{eq:mm}, which implies 
that $d\mu(x)$ is surjective for all $x\in M$. Hence the quotient $\mu^{-1}(0)/K$ carries a natural orbifold structure.

The more general case (for K\"ahler manifolds and semistable actions) was solved by Heinzner and Loose in \cite{he-lo}.

We state the Marsden-Weinstein Theorem now. We assume that $(M,\,\omega_M)$ is a symplectic manifold
on which a compact group $K$ acts with a moment map $\mu\,:\,M \,\longrightarrow\, \k^*$. Assume that $\mu^{-1}(0)$ is nonempty.

\begin{theorem}[Mardsen-Weinstein, \cite{MaWe}]
The set $\mu^{-1}(0)$ is $K$-invariant, and the quotient $\mu^{-1}(0)/K$ possesses a natural
symplectic structure, if the stabilizer of $K$ with respect to all $x\,\in\, \mu^{-1}(0)$ is finite.
\end{theorem}

We will always assume that a reductive group $G$ acts on $M$ with finite stabilizers and that all $G$-orbits intersect $\mu^{-1}(0)$ avoiding categorical quotients.

In our situation the following version holds.

\begin{theorem}[{\cite{kirwan}}]\label{th:kir}
Let $(M,\omega_M)$ be a K\"ahler manifold on which a reductive complex Lie group $G$ with maximal compact subgroup $K$ acts such that $\omega_M$ is $K$-invariant. Assume that all isotropy groups are finite, and suppose the existence of a moment map $\mu \,
:\, M\,\longrightarrow\, \k^*$. Furthermore suppose that all $G$-orbits intersect $\mu^{-1}(0)$.

Then the inclusion $\mu^{-1}(0)\subset M$ induces a diffeomorphism of geometric quotients
$$
\pi\,:\, \mu^{-1}(0)/K \,\longrightarrow\, M/G
$$
such that $M/G$ carries the structure of a complex orbifold and the symplectic form on the symplectic quotient amounts to a K\"ahler orbifold form on $M/G$.
\end{theorem}

\section{Quotients of Sasakian manifolds}

We fix the assumption for our main theorem. Let $(X,g,\xi)$ be a Sasakian manifold, and $(M,\omega_M)$ be the 
associated K\"ahler manifold, i.e.\ $X\times \R_+$ with the induced K\"ahler structure. We assume that a connected 
reductive Lie group $G$ acts on $M$ holomorphically fixing the K\"ahler form $\omega_M$. Let $K\, \subset\, G$ be
a maximal compact subgroup. Assume that this subgroup 
$K$ fixes the subspace $X$; note that this condition implies that the action of $K$ on $X$ preserves the Riemannian metric
$g$.

In particular $\k\,=\, {\rm Lie}(K)$ consists of Killing vector fields on $X$. We assume the existence of a moment 
map for the above group action.

To begin with we make the extra assumption that both $K$ and $G$ act freely. We saw in Section~\ref{se:symq} that 
$\mu^{-1}(0)$ is a differentiable manifold so that the quotient $\mu^{-1}(0)/K$ is also smooth.

Assume that the elements of the Lie algebra $\k$ are perpendicular to the Reeb field $\xi$, and to $\pt/\pt r$:
\begin{eqnarray}
\xi &\perp& \k \quad\text{ with respect\ to } g\text{ on } X\simeq X\times\{1\} \label{eq:xpk}\\
 {\pt}/{\pt r} &\perp& \k \quad\text{ with respect\ to } \omega_M \text{ and } [\k,\,{\pt}/{\pt r}]\,=\,0 \label{eq:kperpr}
\end{eqnarray}

The above assumption \eqref{eq:kperpr} implies that the group $K$ acts on $X\simeq X\times \{1\}$ and all spaces 
$X\times\{r\}$ for all $r\,\in\, {\mathbb R}_+$.

We already know that $K$ fixes $\mu^{-1}(0)$ from the Marsden-Weinstein Theorem. In fact this follows readily from \eqref{eq:mm}.

\begin{lemma}\label{le:R+K}
The action of $\R_+$ on $X\times \R_+$, given by the multiplication of $\R_+$ and the trivial action
of $\R_+$ on $X$, commutes with the action of $K$ and fixes the subset $\mu^{-1}(0)$.
\end{lemma}

\begin{proof}
The first statement follows from $[\k,\,\pt/\pt r]\,=\,0$. Furthermore, because of \eqref{eq:kperpr} we have $d\mu(\pt/\pt r)= \omega_M(\xi,\pt/\pt r)=0$, which shows the second claim.
\end{proof}

We study the compatibility of the action of $K$ on $X\times \R_+$ and the complex structure of the associated K\"ahler structure. We already know that

\begin{lemma}
Concerning the action of $K$ and the almost complex structure $J\vert_F$, the following holds on $X$: For any $u\,
\in\, \k$ and $v\,\in \,\xi^\perp$
\begin{equation}
[u,(J|F)](v) \,=\, -(\nabla_u \Phi)(v) \,=\, -g(u,v)\xi\, .
\end{equation}
\end{lemma}
The claim follows immediately from \eqref{id.}.

We will call the elements of $\k$ also {\em vertical} vector fields, and those perpendicular to $\k$ {\em horizontal}.

\begin{lemma}
On $(X,g,\xi)$ we have
\begin{equation}
[\xi,\, \k]\,=\,0\, .
\end{equation}
\end{lemma}

\begin{proof}
Let $v\in \k$, and $w$ be an arbitrary vector field. Note that $v$ is Killing. Then
\begin{gather*}
g([\xi,v],w)\,=\, g(\nabla_\xi v, w)- g(\nabla_v\xi, w)= -g(\nabla_w v, \xi))- g(\nabla_v\xi, w)\\ = - wg(v,\xi) + g(v,\nabla_w\xi) - g(\nabla_v\xi, w).
\end{gather*}
Now the first term vanishes, because $\xi\perp \k$, and the second and third term together give zero, because $\xi$ is a Killing vector field.
\end{proof}

We also have the action of ${\mathbb R}_+$ by multiplication on the second factor on $M\,=\,X\times \R_+$.

\begin{lemma}\label{le:timesr+}
The group $K$ acts in a free way on the differentiable manifold \break
$(\mu^{-1}(0)\cap (X\times \{1\}))\times \R_+$. The natural bijection
$$
(\mu^{-1}(0)\cap (X \times \{1\}))\times \R_+ \,\longrightarrow\, \mu^{-1}(0) ,\qquad ((x,1),r)\,\longmapsto\, (x,r)
$$
induces an isomorphism
$$
(\mu^{-1}(0)\cap (X \times \{1\}))/K \times \R_+ \,\longrightarrow\, \mu^{-1}(0)/K.
$$
\end{lemma}

\begin{proof}
We consider the surjection $\mu^{-1}(0)\cap (X \times \{1\})\times \R_+ \,\longrightarrow\, \mu^{-1}(0)/K$.
Recall that $X\times\{1\}$ is preserved by the action of $K\, \subset\, G$ on $M$.
Let $(x,r)\,=\, \gamma\cdot (\wt x,s)$ for $x,\, \wt x\,\in\, X \,=\, X\times\{1\}$ and $\gamma \in K$. We use Lemma~\ref{le:R+K} and see that
$$
(x,s^{-1}r)\,=\,s^{-1}\gamma(\wt x,s)\,=\,\gamma s^{-1}(\wt x,s)\,=\,(\gamma \wt x,1)
$$
using the action of $K$ on $X\,=\,X\times\{1\}$. Hence $x=\gamma\wt x$ and $s\,=\,r$.
\end{proof}

We denote the differentiable manifold $\mu^{-1}(0)\cap (X \times \{1\})$ by $\mu^{-1}_X(0)$, where $\mu_X$ stands for the restriction of $\mu$ to $X\times\{1\}$. The quotient manifold $\mu^{-1}_X(0)/K$ is denoted by $Y$, with projection map $\pi\,
:\,\mu^{-1}_X(0)\,\longrightarrow\, Y$.

Since $K$ acts in an isometric way on $X$, the restriction of $g$ to horizontal tangent vectors defines a Riemannian metric $g_Y$ on $Y$. We denote by $\nabla^Y$ corresponding covariant differentiation.

We consider the restriction of $\xi$ to $\mu^{-1}_X(0)$ with values in $TX$, and denote it with the same letter. This vector field is $K$-invariant and orthogonal to $K$-orbits by \eqref{eq:xpk}. Hence it descends to a vector field $\xi_Y$ on $Y$.

We introduce the following notation: Given a vector field $v$ on $Y$, we denote the horizontal lift of $v$ to $\mu^{-1}_X(0)$ by $\wt v$. In this sense $\xi= \wt\xi_Y$.

\begin{lemma}\label{le:nabla}
Let $u,v,w$ be vector fields on $Y$ with horizontal lifts $\wt u,\wt v, \wt w$ to $\mu^{-1}_X(0)$, and let $f\in C^\infty(Y)$. Then
\begin{itemize}
\item[(i)] $\pi^*([u,\,v]f)\,=\, [\wt u,\,\wt v](\pi^*f)$, in particular $[\wt u,\,\wt v]-\wt{[u,\,v]}$ is
tangent to the fibers of $\pi$, and invariant under the action of $K$, hence an element of $\k$.

\item[(ii)] $\wt{\nabla^Y_u v} - \nabla_{\wt u}\wt v \,\in\, \k $.

\item[(iii)] Let $\Phi^Y(u)\,=\, -\nabla^Y_u(\xi_Y)$. Then $\wt{\Phi^Y(u)}-\Phi(\wt u)\,\in\, \k$.
\end{itemize}
\end{lemma}

\begin{proof}
The first identity follows from the definition, the second from (i) and applying the Koszul formula twice showing
$$
g(\nabla_{\wt u}\wt v,\,\wt w) \,=\, \pi^*g_Y(\nabla_u v,\,w)\, .
$$
Now (ii) implies that $\wt{\nabla^Y_w \xi_Y}- \nabla_{\wt w} \xi $ is a vertical vector field.
\end{proof}

Above we defined $Y\,=\,\mu^{-1}_X(0)/K\,=\,(\mu^{-1}(0)\cap(X\times\{1\}))/K$.

\begin{proposition}\label{pr:saq}
The manifold $(Y,g_Y,\xi_Y)$ is Sasakian.
\end{proposition}

\begin{proof}
We note first that the horizontal lift of the Reeb field $\xi_Y$ is equal to the original Reeb field $\xi$ on $X$. We will verify condition (ii) of Theorem~\ref{de:sasaki} on $Y$, using the tilde notation for horizontal lifts. Let $u$, and $v$ be vector fields on $Y$. We will apply Lemma~\ref{le:nabla} repeatedly, and use $\equiv_\k$ for equivalence modulo elements of $\k$.
\begin{gather*}
(\wt{\nabla^Y_u\Phi^Y)(v)} \,=\,\wt{\nabla^Y_u(\Phi^Y(v))} -
\wt{\Phi^Y(\nabla^Y_u)(v))}
\equiv_\k \nabla_{\wt u}(\wt{\Phi^Y(v)})+ \nabla_{\wt{\nabla^Y_u(v)}}(\xi)\\
\equiv_\k -\nabla_{\wt u}\nabla_{\wt v}(\xi) + \nabla_{\nabla_{\wt u}\wt v}(\xi) = (\nabla_{\wt u}\Phi)(\wt v)= g(\wt u,\wt v)\xi - g(\xi,\wt v)\wt u\\
= \wt{g_Y(u,v) \xi_Y} - \wt{g_Y(\xi_Y,v)u} \hspace{5.9cm}
\end{gather*}
Hence
$$
(\nabla^Y_u\Phi^Y)(v)\,=\,{g_Y(u,v) \xi_Y} - {g_Y(\xi_Y,v)u}\, .
$$
\end{proof}

\section{Symplectic reduction for Sasakian manifolds}
Let $(X,g,\xi)$ be a Sasakian manifold, and $(M,\omega_M)$ the associated K\"ahler manifold. Let $G$ be a semisimple complex Lie group acting holomorphically on $M$ with finite stabilizers, fixing $\omega_M$. We know that in this case a moment map $\mu
\,:\, M\,\longrightarrow\, \k^*$ exists. In the somewhat more general case of a reductive group $G$ we make the existence of a moment map an assumption. By Kirwan's result the Marsden-Weinstein symplectic quotient
$$
\mu^{-1}(0)/K \overset{\sim}{\longrightarrow} M/G
$$
possesses a complex structure turning the symplectic form on the quotient into a K\"ahler form.

Our geometric assumptions are \eqref{eq:xpk} and \eqref{eq:kperpr}.

Proposition~\ref{pr:saq} states that $(Y,\,g_Y,\,\xi_Y)$ is a Sasakian manifold, and by Lemma~\ref{le:timesr+} we have
$$
Y \times \R_+ \,\overset{\sim}{\longrightarrow}\, \mu^{-1}(0)/K \overset{\sim}{\longrightarrow} M/G\, .
$$
So far we assumed free group actions. In case of finite stabilizers $Y$ is a Sasakian orbifold and $Y\times \R_+
\,\simeq\, M/G$ is a K\"ahler orbifold. The differential geometric computation remains the same.

\begin{theorem}\label{thm-m}
The geometric K\"ahler quotient $M/G$ is induced by a natural structure of a Sasakian orbifold on the quotient $\mu_X^{-1}(0)/K=\mu^{-1}(0)\cap (X\times\{1\})/K$.
\end{theorem}

\section*{Acknowledgements}

After the paper was written, Liviu Ornea kindly brought to our attention \cite{GO}; we thank him for that.
We thank the referee for going through the paper very carefully.
We thank the International Centre for Theoretical Sciences for hospitality while the work was carried out.
The first-named author is partially supported by a J. C. Bose Fellowship.



\begin{thebibliography}{000}

\bibitem{bai} Baily, W.L.: On the Imbedding of V-Manifolds in Projective Space.
Amer.\ J.\ Math.\ {\bf 79} (1957), 403--430.

\bibitem{BGsusy} Boyer, C.P., Galicki, K.: Sasakian geometry, holonomy and supersymmetry, arXiv:math/0703231.

\bibitem{GO} Grantcharov, G., Ornea, L.: Reduction of Sasakian manifolds, J. Math. Phys. {\bf 42} (2001), 3809--3816.

\bibitem{Gu-St} Guillemin, V., Sternberg, S.: Symplectic techniques in physics. Cambridge University Press. XI (1984).

\bibitem{he-lo} Heinzner, P., Loose, F.: Reduction of complex Hamiltonian G-spaces, Geometric and Functional Analysis {\bf 4} (1994), 288--297.

\bibitem{kirwan} Kirwan, F.C.: Cohomology of Quotients in Symplecnc and Algebraic Geometry, Mathematical Notes, {\bf 31} (1984). Princeton, New Jersey: Princeton University Press. III.

\bibitem{Ma} Maldacena, J.: The large N limit of superconformal field theories and supergravity. Adv. Theor. Math. Phys.
{\bf 2} (1998), 231--252.

\bibitem{MaWe} Marsden, J., Weinstein, A.: Reduction of symplectic manifolds with symmetry. Reports on Math.\ Phys.\ {\bf 5} (1974), 121--130.

\bibitem{Ma1} Martelli, D., Sparks, J.: Notes on toric Sasaki-Einstein seven-manifolds and ${\rm AdS}_4/{\rm CFT}_3$.
J. High Energy Phys. (2008), no. 11, 016.

\bibitem{Ma2} Martelli, D., Sparks, J., Yau, S.-T.: Sasaki-Einstein manifolds and volume minimisation. Comm. Math.
Phys. {\bf 280} (2008), 611--673.

\bibitem{Ma3} Gauntlett, J.P., Martelli, D., Sparks, J., Yau, S.-T.: Obstructions to the existence of Sasaki-Einstein
metrics. Comm. Math. Phys. {\bf 273} (2007), 803--827.

\bibitem{Sa1} Sasaki, S.: On differentiable manifolds with certain structures which are closely related to almost contact
structure. I. T\^ohoku Math. J. {\bf 12} (1960), 459--476.

\bibitem{Sa2} Sasaki, S.: Selected papers. With a foreword by Shiing Shen Chern. Edited by Shun-ichi Tachibana. Kinokuniya
Company Ltd., Tokyo, 1985.

\bibitem{SH} Sasaki, S., Hatakeyama, Y.: On differentiable manifolds with contact metric structures. J. Math. Soc. Japan
{\bf 14} (1962), 249--271.

\bibitem{snow} Snow, D.: Reductive group actions on Stein spaces, Math.\ Ann.\ {\bf 259} (1982), 79--97.

\end{thebibliography}
\end{document}